\documentclass[14pt,a4paper]{article}
\usepackage[utf8]{inputenc}
\usepackage{amsmath}
\usepackage{amsfonts}
\usepackage{amssymb}
\usepackage{amsthm}
\usepackage{floatrow}
\usepackage{caption}
\usepackage{graphicx}
\usepackage{multirow}
\theoremstyle{plain}
\newtheorem{theorem}{Theorem}
\newtheorem{cor}{Corollary}
\newtheorem{lemma}{Lemma}
\newtheorem{utv}{Proposition}

\newtheorem{definition}{Definition}

\usepackage{hyperref}
\newcommand{\footremember}[2]{%
    \footnote{#2}
    \newcounter{#1}
    \setcounter{#1}{\value{footnote}}%
}
 
\title{The Hirzebruch-Mumford covolume of some hermitian lattices}
\author{%
  Stuken Ekaterina\footremember{alley}{Chebyshev Laboratory, St. Petersburg State University, 14th Line V.O., 29B, Saint Petersburg 199178
Russia.

  }}

\begin{document}

\maketitle



\begin{abstract}
Let $L=diag(1,1,\ldots,1,-1)$ and $M=diag(1,1,\ldots,1,-2)$ be the lattices of signature $(n,1)$. We consider the groups $\Gamma=SU(L,\mathcal{O}_K)$ and $\Gamma'=SU(M,\mathcal{O}_K)$ for an imaginary quadratic field $K=\mathbb{Q}(\sqrt{-d})$ of discriminant $D$ and it's ring of integers $\mathcal{O}_K$, $d$ odd and square free. We compute the Hirzebruch-Mumford volume of the factor spaces $\mathbb{B}^n/\Gamma$ and $\mathbb{B}^n/\Gamma'$. The result for the factor space $\mathbb{B}^n/\Gamma$ is due to Zeltinger \cite{Zel}, but as we're using it to prove the result for $\mathbb{B}^n/\Gamma'$ and it is hard to find his article, we prove the first result here as well.
\end{abstract}

\section*{Introduction}
Let $L$ be the lattice $diag(1,1,\ldots,1,-1)$ of signature $(n,1)$. We consider the group $\Gamma=SU(L,\mathcal{O}_K)$ for the ring of integers $\mathcal{O}_K$ of an imaginary quadratic field $K=\mathbb{Q}(\sqrt{-d})$ of discriminant $D$, $d$ odd and square free. Denote by $Vol_{HM}(\mathbb{B}^n/\Gamma)$ the Hirzebruch-Mumford volume of the factor space $\mathbb{B}^n/\Gamma$. We denote by $L(k)$ the L-function of the quadratic field $K$ with character $\chi_D(p)=\left(\frac{D}{p}\right)$ (Kronecker symbol). We prove the following theorem:
\begin{theorem}\label{L}
The Hirzebruch-Mumford volume of the factor space $\mathbb{B}^n/\Gamma$ equals
\begin{table}[h]
\centering
	\begin{tabular}{|c|c|c|}
		\hline
		{\bf n} & {\bf D} & {\bf $Vol_{HM}(\mathbb{B}^n/\Gamma)$} \\
	 	\hline
	 	{\bf even} &  {\bf -4d } & {\bf $D^{\frac{n^2+3n}{4}}\cdot \prod\limits_{j=1}^{n}\frac{j!}{(2\pi)^{j+1}}\cdot \zeta(2)\cdot L(3)\cdot \zeta(4)\cdot L(5)\cdot \ldots \cdot L(n+1)$}   \\
	 	\hline
	 	{\bf even} &  {\bf -d } & {\bf $D^{\frac{n^2+3n}{4}}\cdot \prod\limits_{j=1}^{n}\frac{j!}{(2\pi)^{j+1}}\cdot \zeta(2)\cdot L(3)\cdot \zeta(4)\cdot L(5)\cdot \ldots \cdot L(n+1)$}   \\
	 	\hline
	 	{\bf odd} &  {\bf -4d} & {\bf $D^{\frac{n^2+3n}{4}}\cdot(1-2^{-(n+1)})\cdot \prod\limits_{p|d} (1+\left( \frac{(-1)^{(n+3)/2}}{p}\right)\cdot p^{-\frac{n+1}{2}}) \prod\limits_{j=1}^{n}\frac{j!}{(2\pi)^{j+1}}\cdot \zeta(2)\cdot  \ldots \cdot \zeta(n+1)$}   \\
	 	\hline
	 	{\bf odd} &  {\bf -d} & {\bf $D^{\frac{n^2+3n}{4}}\cdot \prod\limits_{p|d} (1+\left( \frac{(-1)^{(n+3)/2}}{p}\right)p^{-\frac{n+1}{2}}) \prod\limits_{j=1}^{n}\frac{j!}{(2\pi)^{j+1}}\cdot \zeta(2)\cdot L(3)\cdot  \ldots \cdot \zeta(n+1)$}      \\
	 	\hline
	\end{tabular}
\end{table}
\end{theorem}
\vspace{-0.5cm}

Using theorem \ref{L}, we compute the $Vol_{HM}(\mathbb{B}^n/\Gamma')$ for the group $\Gamma'=SU(M,\mathcal{O}_K)$, where the lattice $M=diag(1,1,\ldots,1,-2)$, and prove the following theorem:

\begin{theorem}\label{M}
The Hirzebruch-Mumford volume of the factor space $\mathbb{B}^n/\Gamma'$ equals
\begin{table}[h]
\centering
	\begin{tabular}{|c|c|c|}
		\hline
		{\bf n} & {\bf D} & {\bf $Vol_{HM}(\mathbb{B}^n/\Gamma')$} \\
	 	\hline
	 	{\bf even} &  {\bf -4d } & {\bf $D^{\frac{n^2+3n}{4}}\cdot \prod\limits_{j=1}^{n}\frac{j!}{(2\pi)^{j+1}}\cdot \zeta(2)\cdot L(3)\cdot \zeta(4)\cdot L(5)\cdot \ldots \cdot L(n+1)\cdot (2^n-1)$}   \\
	 	\hline
	 	{\bf even} &  {\bf -d} & {\bf $D^{\frac{n^2+3n}{4}}\cdot \prod\limits_{j=1}^{n}\frac{j!}{(2\pi)^{j+1}}\cdot \zeta(2)\cdot L(3)\cdot \zeta(4)\cdot L(5)\cdot \ldots \cdot L(n+1)\cdot 2^n\cdot \frac{1-\left(\frac{-d}{2}\right)^{n+1}2^{-(n+1)} }{1-\left( \frac{-d}{2} \right)2^{-1}}$}   \\
	 	\hline
	 	{\bf odd} &  {\bf -4d} & {\bf $D^{\frac{n^2+3n}{4}}\cdot(1-2^{-(n+1)})\cdot \prod\limits_{p|d} (1+\left( \frac{(-1)^{(n+3)/2}\cdot 2}{p}\right)\cdot p^{-\frac{n+1}{2}}) \prod\limits_{j=1}^{n}\frac{j!}{(2\pi)^{j+1}}\cdot 2^n \zeta(2)  \ldots \zeta(n+1)$}   \\
	 	\hline
	 	{\bf odd} &  {\bf -d} & {\bf $D^{\frac{n^2+3n}{4}}\cdot \prod\limits_{p|d} (1+\left( \frac{(-1)^{(n+3)/2}\cdot 2}{p}\right)p^{-\frac{n+1}{2}}) \prod\limits_{j=1}^{n}\frac{j!}{(2\pi)^{j+1}}\cdot 2^n\frac{1-\left(\frac{-d}{2}\right)^{n+1}2^{-(n+1)} }{1-\left( \frac{-d}{2} \right)2^{-1}}\zeta(2)  \ldots  \zeta(n+1)$}      \\
	 	\hline
	\end{tabular}
\end{table}
\end{theorem}

\section*{Acknowledgments} The author is very grateful to O.V. Schwarzman for suggesting the problem and numerous valuable conversations. The author was supported by RSF grant no. 19-71-30002 and by Chebyshev Laboratory, St. Petersburg State University.

\section*{Local measures $\tau_p$} In this section we remind the notion of a local measure and some well-known facts about local measures.
Consider the group $G=SU(L,K)$. More explicitly, $G$ is the group of matrices $A$ such that $\begin{cases} AL\bar{A}'=L, \\ \det(A)=1.\end{cases}$ The Lie algebra $\mathcal{L}$ of the group $G$ is defined by the system $\begin{cases} BL+L\bar{B}'=0,\\ Tr(B)=0.\end{cases}$ Let $\mathcal{L}_\mathbb{Z}$ be the elements of $\mathcal{L}$ with coefficients in $\mathcal{O}_K$. Consider the matrix norm $||A||=max|a_{ij}|_p$ in the Lie algebra $\mathcal{L}_p=\mathcal{L}_{\mathbb{Q}}\otimes \mathbb{Q}_p$. Let $B(k)=\{ A, ||A||<\frac{1}{p^k}\}$ be a ball in $\mathcal{L}_p$ of a sufficiently small radius. Denote by $G_{\mathbb{Z}_p}^{(k)}$ the image of $B(k)$ under the exponential map. The following facts are well-known and can be found, for example, in \cite{Shv}:
\begin{utv}\label{tau_p from balls}
\textbf{a) }The group $G_{\mathbb{Z}_p}^{(k)}$ is a congruence subgroup of $G_{\mathbb{Z}_p}$ modulo $p^k$ for $k\geq 1$ and $p\neq 2$ or $k\geq 2$ and $p=2$.\\
\textbf{b) }The manifold $G_{\mathbb{Z}_p}$ splits into a disjoint union of $[G_{\mathbb{Z}_p}:G_{\mathbb{Z}_p}^{(k)}]$ balls of the same volume.\\
\textbf{c) } The volume of $G_{\mathbb{Z}_p}^{(k)}$ is equal to the volume of $B(k)$ starting from $k=1$ for $p\neq 2$ and from $k=2$ for $p=2$. 
\end{utv}

It follows that $$Vol_{\tau_p} (G_{\mathbb{Z}_p})= \begin{cases} \dfrac{N_p}{p^{(n+1)^2-1}},~ p\neq 2;\\
\dfrac{N_2}{4^{(n+1)^2-1}}, ~p=2,
 \end{cases}$$ where $N_p$ is the number of elements in $G_{\mathbb{Z}_p}/G_{\mathbb{Z}_p}^{(1)}$ for $p\neq 2$ and the order of $G_{\mathbb{Z}_2}/G_{\mathbb{Z}_2}^{(2)}$ for $p=2$. We need the following version (\cite{Shv}) of Hensel's lemma to compute $N_p$:
 \begin{theorem}\label{Hensel}
 Let $f=(f_1,\ldots, f_r)$ be a system of $r$ elements, $f_i\in \mathbb{Z}_p[x_1,\ldots,x_n]$. Denote by $M_f(a)$ the Jacobi matrix of this system in $a\in \mathbb{Z}_p^n$. Let $e$ be the greatest invariant factor of $M_f(a)$. Suppose that $f(a)=0 \pmod{e^2p}$. Then there exists a solution $\tilde{a}\in \mathbb{Z}_p^n$ of the system $f=0$ such that $\tilde{a}\equiv a\pmod{ep}$.
\end{theorem}
\begin{cor}\label{Hensel_cor}
Let $U_n$ be the set of solutions of the system $f=0$ in the ring $\mathbb{Z}_p/e^n p\mathbb{Z}_p$. Consider the set $U$ of solutions of the system $f=0$ in $\mathbb{Z}_p$ and the reduction map $\phi:U\to U_1$. Consider also the map $\phi_1:U_2\to U_1$. Then $\phi(U)=\phi_1(U_2)$. 
\end{cor}
One can easily check, that in our case the invariant factors are $(1,1,\ldots,2,2)$. Then it follows from the corollary that $N_p=\# SU(L,\mathcal{O}/p\mathcal{O})$ and $N_2=\# Im(\phi_1)=\dfrac{\# SU(L,\mathcal{O}/2^3\mathcal{O})}{\# ker(\phi_1)}$.

So, $$Vol_{\tau_p} (G_{\mathbb{Z}_p})= \begin{cases} \dfrac{\# SU(L,\mathcal{O}/p\mathcal{O})}{p^{(n+1)^2-1}},~ p\neq 2;\\
\dfrac{\# SU(L,\mathcal{O}/2^3\mathcal{O})}{4^{(n+1)^2-1}\cdot \# ker(\phi_1)}, ~p=2,
 \end{cases}$$

The following theorem is well-known (\cite{Wei}):
\begin{theorem}
The Tamagawa number of the group $SU$ equals 1.
\end{theorem}

It means that $\tau_{\infty}(G_{\mathbb{R}}/G_{\mathbb{Z}})=\dfrac{1}{\prod\limits_{p\neq \infty}\tau_p(G_{\mathbb{Z}_p})}$.


\section*{Hirzebruch-Mumford volume computation}
Consider the bilinear form $B=Tr(XY)$ on the Lie algebra $\mathcal{L}$. Let $\mathcal{L}=k\oplus m$ be the Cartan decomposition. Denote by $\omega_1$ the volume form on $k$ and by $\omega_2$ the volume form on $m$ constructed by the bilinear form $B$. Let $\omega^{(d)}=\omega_1 \cdot \omega_2$. 

\begin{lemma}\label{omega_d}  $|Vol_{\omega^{(d)}}(G_{\mathbb{R}}/G_{\mathbb{Z}})|=d^{\frac{n(n+3)}{4}}\sqrt{n+1}|Vol_{\tau_\infty}(G_{\mathbb{R}}/G_{\mathbb{Z}})|$ for $d\equiv 3 \pmod{4}$ and $|Vol_{\omega^{(d)}}(G_{\mathbb{R}}/G_{\mathbb{Z}})|=d^{\frac{n(n+3)}{4}}\cdot 2^{\frac{n(n+1)}{2}}\sqrt{n+1}|Vol_{\tau_\infty}(G_{\mathbb{R}}/G_{\mathbb{Z}})|$ for $d\equiv 1\pmod{4}$
\end{lemma}
\begin{proof} Let $\mathcal{O}_K=\mathbb{Z}\cdot 1 + \mathbb{Z}\cdot \epsilon$.
We fix the basis of $\mathcal{L}_\mathbb{Z}$ over $\mathbb{Z}$:
$e_k=\begin{pmatrix}
0& \ldots & 0 & \ldots& 0\\
\vdots& & \vdots &  & \vdots \\
0 & \ldots & 0&  \ldots & \epsilon\\
\vdots& &\vdots & &\vdots \\
0 & \ldots& \bar{\epsilon} & \ldots & 0
\end{pmatrix}$  (non-zero elements are in the $k$-th row of $n+1$-th column and $k$-th column of $n+1$-th row, $k\in\{1...n\}$);
$f_k=\begin{pmatrix}
0& \ldots & 0 & \ldots& 0\\
\vdots& & \vdots &  & \vdots \\
0 & \ldots & 0&  \ldots & 1\\
\vdots& &\vdots & &\vdots \\
0 & \ldots& 1 & \ldots & 0
\end{pmatrix}$  (non-zero elements are in the $k$-th row of $n+1$-th column and $k$-th column of $n+1$-th row, $k\in\{1...n\}$);
$g_k=\begin{pmatrix}
0& \ldots & 0 & 0\\
\vdots& & \vdots & \vdots  \\
0 & \ldots & \sqrt{-d} & 0\\
0 & \ldots& 0 &  -\sqrt{-d}
\end{pmatrix}$ (non-zero elements in $k$-th and $k+1$-th rows, $k\in\{1...n\}$); 
$e_{i,j}=\begin{pmatrix}
0& \ldots & 0 & &0&\ldots &0\\
\vdots& & \vdots & &\vdots & &\vdots  \\
0 & \ldots & 0 & \ldots & \epsilon & \ldots &0\\
\vdots& & \vdots & &\vdots & & \vdots \\
0 & \ldots& -\bar{\epsilon} & \ldots & 0&\ldots &0\\
\vdots& & \vdots & &\vdots & & \vdots \\
0& \ldots & 0 & & 0 & \ldots & 0
\end{pmatrix}$ (non-zero elements on the intersection of $i$-th row and $j$-th column and $j$-th row and $i$-th column, $i,j\in\{1...n\}, ~i\neq j$);
$f_{i,j}=\begin{pmatrix}
0& \ldots & 0 & &0&\ldots &0\\
\vdots& & \vdots & &\vdots & &\vdots  \\
0 & \ldots & 0 & \ldots & 1 & \ldots &0\\
\vdots& & \vdots & &\vdots & & \vdots \\
0 & \ldots& -1 & \ldots & 0&\ldots &0\\
\vdots& & \vdots & &\vdots & & \vdots \\
0& \ldots & 0 & & 0 & \ldots & 0
\end{pmatrix}$ (non-zero elements on the intersection of $i$-th row and $j$-th column and $j$-th row and $i$-th column, $i,j\in\{1...n\}, ~i\neq j$). In our basis $\tau_{\infty}=\bigwedge e_k \bigwedge f_k \bigwedge g_k \bigwedge e_{i,j}\bigwedge f_{i,j}$.
The matrix of the Killing form in the basis  $\{ g_1,\ldots, g_n,e_{1,2}, f_{1,2},\ldots e_{n-1,n},f_{n-1,n},e_1, f_1,\ldots e_n, f_n\}$ is:\\

\[
  B_L=\left(\begin{array}{@{}ccccc|cc|c|cc|c@{}}
    -2d & d & 0 & 0 &\ldots & 0 & 0&\ldots &0 &0  &\ldots\\
    d & -2d & d & 0 & \ldots& 0 & 0 &\ldots &0 &0 &\ldots\\
    0 & d& -2d& d&\ldots&  0& 0& \ldots& 0& 0&\ldots\\ 
    \vdots&  \vdots & \vdots & \vdots &\ddots   & \vdots& \vdots&\ddots & \vdots& \vdots& \ddots\\\hline
    0 & 0 & 0 & 0 & \ldots & -2\epsilon \bar{\epsilon} & -\epsilon-\bar{\epsilon} &\ldots &0 &0 &\ldots \\
    0 & 0 & 0 & 0 & \ldots & -\epsilon -\bar{\epsilon} & -2 & \ldots & 0& 0&\ldots\\ \hline
    \vdots & \vdots & \vdots & \vdots  & \ddots & \vdots & \vdots&\ddots & \vdots& \vdots& \ddots \\ \hline
    0 & 0 & 0 & 0 & \ldots & 0& 0& \ldots& 2\epsilon \bar{\epsilon} & \epsilon+\bar{\epsilon} & \ldots\\
    0 & 0 & 0 & 0 & \ldots & 0& 0& \ldots & \epsilon +\bar{\epsilon} & 2 & \ldots\\ \hline
    \vdots  & \vdots & \vdots & \vdots & \ddots & \vdots & \vdots& \ddots&\vdots &\vdots & \ddots  \\ \hline
  \end{array}\right)
\]
 So $\sqrt{|\det(B)|}=d^{\frac{n}{2}}(\epsilon-\bar{\epsilon})^{\frac{n(n+1)}{2}}\sqrt{n+1}$. Clearly, $\sqrt{|\det(B)|}=d^{\frac{n(n+3)}{4}}\sqrt{n+1}$ for $\epsilon=\frac{1+\sqrt{-d}}{2}$ and $\sqrt{|\det(B)|}=d^{\frac{n(n+3)}{4}}\cdot 2^{\frac{n(n+1)}{2}}\sqrt{n+1}$ for $\epsilon=\sqrt{-d}$. We get the statement of the lemma.
\end{proof}

Consider a closed Riemannian manifold $M^{2s}$ of dimension $2s$. Let $\Omega$ be the Euler-Poincare measure on it. The following theorem is well-known:

\begin{theorem}\cite{Che}
Let $\chi(M^{2s})$ be the Euler characteristic of a closed Riemannian manifold $M^{2s}$. Then 
$\int_{M^{2s}}\Omega=\chi(M^{2s})$. Moreover, if $M^{2s}$ is a homogeneous manifold of constant holomorphic curvarute $k$, then $\Omega=\dfrac{(s+1)! k^s}{(4\pi)^s}dv$, where $dv$ is a volume element corresponding to the Riemannian metric.
\end{theorem}

\begin{lemma}\label{curv}
The curvature $k$ of space $G_{\mathbb{R}}/G_{\mathbb{Z}}$ in metrics $B|_m$ equals $-2$.
\end{lemma}
\begin{proof}
The curvature $k$ can be computed by the formula $$ k=\dfrac{B([[X,JX],X],JX)}{B(X,X)\cdot B(JX,JX)},$$ where $J$ is a complex structure on $\mathbb{B}^n$, $J\begin{pmatrix}
0& \ldots & 0 & a_1\\
\vdots&  & \vdots& \vdots \\
0 & \ldots  & 0 &a_n\\
\bar{a}_1 & \ldots& \bar{a}_n & 0
\end{pmatrix} = \begin{pmatrix}
0& \ldots & 0 & ia_1\\
\vdots&  & \vdots& \vdots \\
0 & \ldots  & 0 &ia_n\\
-i\bar{a}_1 & \ldots& -i\bar{a}_n & 0
\end{pmatrix}$. 
Substituting, for example, basis vector $f_n$ as $X$ in the above formula, we get the statement of the lemma.
\end{proof}

\begin{definition}
The quotient $Vol_{HM}(\mathbb{B}^n/\Gamma)=\dfrac{\chi(\mathbb{B}^n/\Gamma)}{\chi(\mathbb{P}^n(\mathbb{C}))}$ is called the Hirzebruch-Mumford volume of the quotient space $\mathbb{B}^n/\Gamma$ or the Hirzebruch-Mumford covolume of the group $\Gamma$.
\end{definition}

It means that $Vol_{HM}(\mathbb{B}^n/\Gamma)=\dfrac{Vol_{\Omega}(\mathbb{B}^n/\Gamma)}{n+1}$, since $\chi(\mathbb{P}^n(\mathbb{C})) = n+1$.  

So, $$Vol_{HM}(\mathbb{B}^n/\Gamma)=\dfrac{Vol_{\Omega}(\mathbb{B}^n/\Gamma)}{n+1}=Vol_{\omega_2}(\mathbb{B}^n/\Gamma)\cdot\dfrac{n!}{(2\pi)^n}.$$

Let $K$ be the maximal compact subgroup $S(U(n)\times U(1))$ of $G_\mathbb{R}$.
 
We note that $$Vol_{\omega}(G_{\mathbb{R}}/G_{\mathbb{Z}})=Vol_{\omega_1}(K)\cdot Vol_{\omega_2}(K\backslash G / G_{\mathbb{Z}})=$$ $$=Vol_{\omega_1}(K)\cdot Vol_{\omega_2}(\mathbb{B}^n / G_{\mathbb{Z}})=Vol_{\omega_1}(K)\cdot Vol_{HM}( \mathbb{B}^n/ \Gamma)\cdot \dfrac{(2\pi)^n}{n!}.$$

So, $$Vol_{HM}(\mathbb{B}^n/\Gamma)=\dfrac{n! \cdot Vol_{\omega}(G_{\mathbb{R}}/ G_{\mathbb{Z}})}{(2\pi)^n\cdot Vol_{\omega_1}(K)}.$$ 

Using lemma \ref{omega_d} we obtain:

\begin{lemma}\label{itog omega} $$Vol_{HM}(\mathbb{B}^n/\Gamma)=\dfrac{n! \cdot D^{\frac{n^2+3n}{4}}\sqrt{n+1} \cdot Vol_{\tau_\infty}(G_{\mathbb{R}}/G_{\mathbb{Z}})}{(2\pi)^n\cdot Vol_{\omega_1}(K)}, \textit{ for } d\equiv 3\pmod{4}$$
and $$Vol_{HM}(\mathbb{B}^n/\Gamma)=\dfrac{n! \cdot D^{\frac{n^2+3n}{4}}\sqrt{n+1} \cdot Vol_{\tau_\infty}(G_{\mathbb{R}}/G_{\mathbb{Z}})}{(4\pi)^n\cdot Vol_{\omega_1}(K)}, \textit{ for } d\equiv 1\pmod{4}.$$
\end{lemma}

So it requires to compute $Vol_{\omega_1}(K)$ and $Vol_{\tau_\infty}(G_{\mathbb{R}}/G_{\mathbb{Z}})$.

\section*{Computation of $\tau_\infty$ for $L$ uning local measures $\tau_p$}

\begin{lemma}\label{index}
The index $[U(L,\mathcal{O}_K/ p^k \mathcal{O}_K):SU(L,\mathcal{O}_K/ p^k \mathcal{O}_k)]$ equals $2p^k$ if $p$ is ramified and $p^k(1-\left( \frac{D}{p}\right) p^{-1})$ if $p$ is unramified.
\end{lemma}
\begin{proof}
We consider the exact sequence $$1\to SU(L,\mathcal{O}_K/ p^k \mathcal{O}_K)\to U(L,\mathcal{O}_K/ p^k \mathcal{O}_K)\to U(1,\mathcal{O}_K/ p^k \mathcal{O}_K)\to 1.$$ Now it suffices to compute $\#U(1,\mathcal{O}_K/ p^k \mathcal{O}_K)$ which has been done, for example, in \cite{Otr}. 
\end{proof}

\begin{lemma}\label{general_p}
The local volume $\tau_p(G_{\mathbb{Z}_p})$ for $p\nmid D$ equals $$ \tau_p(G_{\mathbb{Z}_p})=\prod\limits_{i=2}^{n+1}\left( 1-\left( \frac{D}{p}\right)^{i}p^{-i} \right).$$
\end{lemma}
\begin{proof}
We compute $\beta_p$ using the article \cite{Gan}. By definition, $\beta_p=\lim\limits_{N\to\infty}p^{-N(n+1)^2}\#U(L,\mathcal{O}_K/p^N\mathcal{O}_K)$. This limit stabilizes for $N\geq 2\nu_p(\det(L))+1$, so in our case $N=1$. By lemma \ref{index} $[U(L,\mathcal{O}_K/ p \mathcal{O}_K):SU(L,\mathcal{O}_K/ p \mathcal{O}_K)]=p(1-\left( \frac{D}{p}\right) p^{-1})$, so $\#SU(L, \mathcal{O}_K/p\mathcal{O}_K)=\dfrac{\beta_p\cdot p^{(n+1)^2-1}}{1-\left( \frac{D}{p}\right) p^{-1}}$. So it suffices to compute $\beta_p$. The Jordan decomposition of the lattice $L$ contains only one component $L_0$, $d_0=0$ (using the notation from \cite{Gan}), so $\beta_p=p^{-\dim G_0}\#G_0(\mathbb{F}_p)$ ($\mathbb{F}_p$ -- finite field of order $p$). Here $G_0=U(n+1)$ if $\chi_D(p)=-1$ and $G_0=GL(n+1)$ if $\chi_D(p)=1$. So, $\dim G_0 = (n+1)^2$. It is well-known (\cite{Art}) that $\#GL(n+1,\mathbb{F}_p)=p^{(n+1)^2}\prod\limits_{i=1}^{n+1}(1-p^{-i})$ and
$\#U(n+1,\mathbb{F}_p)=p^{(n+1)^2}\prod\limits_{i=1}^{n+1}(1-(-1)^{i}p^{-i})$. 
So $$\tau_p(G_{\mathbb{Z}_p})=\dfrac{\#SU(L, \mathcal{O}_K/p\mathcal{O}_K)}{p^{(n+1)^2-1}}=\dfrac{\beta_p}{1-\left( \frac{D}{p}\right) p^{-1}}=\prod\limits_{i=2}^{n+1}\left( 1-\left( \frac{D}{p}\right)^{n+1}p^{-(n+1)} \right).$$
\end{proof}

\begin{lemma} \label{p div D}
The local volume $\tau_p(G_{\mathbb{Z}_p})$ for odd $p$ such that $p \mid D$ equals 
$$\tau_p(G_{\mathbb{Z}_p})=\left(1-\left( \frac{(-1)^{\frac{n+3}{2}}}{p}\right)p^{-\frac{n+1}{2}}\right)\prod\limits_{i=1}^{(n-1)/2}(1-p^{-2i}) \textit{ for odd n},$$ $$\tau_p(G_{\mathbb{Z}_p})=\prod\limits_{i=1}^{n/2}(1-p^{-2i}) \textit{ for even n}.$$ 
\end{lemma}
\begin{proof}
We compute $\beta_p$ using the article \cite{Gan}. By definition, $\beta_p=\lim\limits_{N\to\infty}p^{-N(n+1)^2}\#U(L,\mathcal{O}_K/p^N\mathcal{O}_K)$. This limit stabilizes for $N\geq 2\nu_p(\det(L))+1$, so in our case $N=1$. By lemma \ref{index} $[U(L,\mathcal{O}_K/ p \mathcal{O}_K):SU(L,\mathcal{O}_K/ p \mathcal{O}_K)]=2p$, so $\#SU(L, \mathcal{O}_K/p\mathcal{O}_K)=\dfrac{\beta_p\cdot p^{(n+1)^2}}{2p}$. So it suffices to compute $\beta_p$. The Jordan decomposition of the lattice $L$ contains only one component $L_0$, $d_0=0$ (using the notation from \cite{Gan}), so $\beta_p=p^{-\dim G_0}\#G_0(\mathbb{F}_p)$ ($\mathbb{F}_p$ -- finite field of order $p$). Here $G_0=O(n+1)$, so $\dim G_0=\frac{n(n+1)}{2}$. It is well-known \cite{Art} that $\#O(n+1,\mathbb{F}_p)=2p^{\frac{n(n+1)}{2}}\prod\limits_{i=1}^{\frac{n}{2}}(1-p^{-2i})$ if $n$ is even and
$\#O(n+1,\mathbb{F}_p)=2p^{\frac{n(n+1)}{2}}(1-\epsilon p^{-\frac{n+1}{2}})\prod\limits_{i=1}^{\frac{n-1}{2}}(1-p^{-2i})$ if $n$ is odd. Here $\epsilon =1$ if the discriminant of $L$ is $(-1)^{\frac{n+1}{2}}$ and $\epsilon=-1$ otherwise, so $\epsilon=\left( \frac{(-1)^{\frac{n+3}{2}}}{p}\right)$. 
We get $$\tau_p(G_{\mathbb{Z}_p})=\dfrac{\#SU(L, \mathcal{O}_K/p\mathcal{O}_K)}{p^{(n+1)^2-1}}=\dfrac{\beta_p}{2}=\left(1-\left( \frac{(-1)^{\frac{n+3}{2}}}{p}\right)p^{-\frac{n+1}{2}}\right)\prod\limits_{n=1}^{(n-1)/2}(1-p^{-2i})$$ for odd $n$ and $$\tau_p(G_{\mathbb{Z}_p})=\prod\limits_{n=1}^{n/2}(1-p^{-2i})$$ for even $n$.
\end{proof}

\begin{lemma}\label{p=2}
The local volume $\tau_2(G_{\mathbb{Z}_2})$ equals $\tau_2(G_{\mathbb{Z}_2})=\prod\limits_{i=2}^{n+1}(1-\left( \frac{-d}{2}\right)^i 2^{-i})$ if $2\nmid D$, $\tau_2(G_{\mathbb{Z}_2})=2^{-n}\prod\limits_{i=1}^{[n/2]}(1-2^{-2i})$ if $2\mid D$.
\end{lemma}
\begin{proof}
The computations in case $2\nmid D$ have already been done in lemma \ref{general_p}, so we need only to prove the lemma for $2\mid D$. We compute $\beta_2$ using the article \cite{Cho}. By definition, $\beta_2=\lim\limits_{N\to\infty}p^{-N(n+1)^2}\#U(L,\mathcal{O}_K/2^N\mathcal{O}_K)$. This limit stabilizes for $N\geq 2\nu_p(\det(L))+3$ (because $p=2$ is ramified in our case), so $N=3$. By lemma \ref{index} the index $[U(L,\mathcal{O}_K/ 2^3 \mathcal{O}_K):SU(L,\mathcal{O}_K/ 2^3 \mathcal{O}_K)]=2^4$, so $\#SU(L, \mathcal{O}_K/2^3\mathcal{O}_K)=\beta_2\cdot p^{3(n+1)^2-4}$. So it suffices to compute $\beta_2$. The Jordan decomposition of the lattice $L$ contains only one component $L_0$ of type $I^0$ when $n$ is even and of type $I^e$ when $n$ is odd. Moreover, $d_0=0$, $N=0$, $\beta=1$, $f=2$ (using the notation from \cite{Cho}). So $\beta_2=2^{-(n+1)^2}\cdot 2^{m+1} \cdot \#Sp(n,2)$ for even $n$ and $\beta_2=2^{-(n+1)^2}\cdot 2^{m+1} \cdot \#Sp(n-1,2)$ for odd $n$. Considering that $m=(n+1)^2-\frac{n(n+1)}{2}$ for even $n$ and $m=(n+1)^2-\frac{n(n-1)}{2}$ for odd $n$ and that $\# Sp(n,2)=2^{\frac{n(n+1)}{2}}\prod\limits_{i=1}^{[n/2]}(1-2^{-2i})$, we get $\beta_2=2\prod\limits_{i=1}^{[n/2]}(1-2^{-2i})$. Now we need to compute $|ker (\phi)|$, where $\phi:SU(L, \mathcal{O}_K/2^3\mathcal{O}_K)\to SU(L, \mathcal{O}_K/2^2\mathcal{O}_K)$. If $A\in \ker(\phi)$ then $A=E+4B$, where $B$ is a matrix with coefficients in $\mathcal{O}/2\mathcal{O}$. The matrix $B$ must satisfy the system: $\begin{cases} -BL=L\bar{B}', \\ Tr(B)=0.\end{cases}$ If $B=(b_{ij})$, we get the following equations: $\begin{cases} -b_{ii}=\bar{b}_{ii}, ~i\in\{1;n\} \\  -b_{ij}=\bar{b}_{ji}, ~i,j\in\{1;n\}, i\neq j\\
b_{i,n+1}=\bar{b}_{n+1,i},~i\in\{1;n\}.
\end{cases}$
 Each equation for $b_{ii}$ has $4$ solutions, and in the remaining equations one summand is uniquely determined by the other one, so it gives $4$ options for each pair $(b_{ij}, ~b_{ji})$. So $|ker(\phi)|=4^n\cdot 4^{\frac{(n+1)^2-(n+1)}{2}}=2^{n^2+3n}$.

We get $$\tau_2(G_{\mathbb{Z}_2})=\dfrac{\#SU(L, \mathcal{O}_K/2^3\mathcal{O}_K)}{4^{(n+1)^2-1}\cdot |ker(\phi)|}=\dfrac{2^{3(n+1)^2-3}\prod\limits_{i=1}^{[n/2]}(1-2^{-2i})}{2^{2n^2+4n}\cdot 2^{n(n+3)}}=2^{-n}\prod\limits_{i=1}^{[n/2]}(1-2^{-2i}).$$
\end{proof}

The values of $\tau_{\infty}(G_{\mathbb{R}}/G_{\mathbb{Z}})$ are written in the table below. To compute them we used the results of lemma \ref{general_p}, lemma \ref{p div D} and lemma \ref{p=2} and the fact that  $\tau_{\infty}(G_{\mathbb{R}}/G_{\mathbb{Z}})=\dfrac{1}{\prod\limits_{p\neq \infty}\tau_p(G_{\mathbb{Z}_p})}$.

\begin{table}[h]
\centering
	\begin{tabular}{|c|c|c|}
		\hline
		{\bf n} & {\bf D} & {\bf $\tau_{\infty}(G_{\mathbb{R}}/G_{\mathbb{Z}})$} \\
	 	\hline
	 	{\bf even} &  {\bf -4d } & {\bf $2^n \zeta(2)\cdot L(3)\cdot \zeta(4)\cdot L(5)\cdot \ldots \cdot L(n+1)$}   \\
	 	\hline
	 	{\bf even} &  {\bf -d } & {\bf $\zeta(2)\cdot L(3)\cdot \zeta(4)\cdot L(5)\cdot \ldots \cdot L(n+1)$}   \\
	 	\hline
	 	{\bf odd} &  {\bf -4d} & {\bf $2^n(1-2^{-(n+1)})\cdot \prod\limits_{p|d} (1+\left( \frac{(-1)^{(n+3)/2}}{p}\right)\cdot p^{-\frac{n+1}{2}}) \cdot \zeta(2)\cdot  \ldots \cdot \zeta(n+1)$}   \\
	 	\hline
	 	{\bf odd} &  {\bf -d} & {\bf $\prod\limits_{p|d} (1+\left( \frac{(-1)^{(n+3)/2}}{p}\right)p^{-\frac{n+1}{2}})  \zeta(2)\cdot L(3)\cdot  \ldots \cdot \zeta(n+1)$}      \\
	 	\hline
	\end{tabular}
\end{table}

\section*{The volume $Vol_{\omega_1}(K)$ of the maximal compact group $K=S(U(n)\times U(1))$}
We will need the following result.
\begin{lemma} $Vol_{\omega_1}(SU(n))=\dfrac{\sqrt{n}(2\pi)^{\frac{n^2+n-2}{2}}}{\prod\limits_{i=1}^{n-1}i!}$.
\end{lemma}
\begin{proof}
It is well-known (\cite{Mac}) that if we take the measure to be $tr(ad_X ad_Y)$, then the volume of the group $SU(n)$ equals $\dfrac{2^{\frac{n^2-1}{2}}n^{\frac{n^2}{2}}(2\pi)^{\frac{n^2+n-2}{2}}}{1!2!3!\ldots (n-1)!}$. In the Lie algebra $su(n)$ the form $tr(ad_X ad_Y)$ corresponds to the form $2n \cdot tr(XY)=2n\cdot B(X,Y)$, so the proportionality coefficient is $(2n)^{\frac{n^2-1}{2}}$. So,
 $$Vol_{\omega_1}(SU(n))=\dfrac{2^{\frac{n^2-1}{2}}n^{\frac{n^2}{2}}(2\pi)^{\frac{n^2+n-2}{2}}}{1!2!3!\ldots (n-1)!(2n)^{\frac{n^2-1}{2}}}=\dfrac{\sqrt{n}(2\pi)^{\frac{n^2+n-2}{2}}}{\prod\limits_{i=1}^{n-1}i!}.$$
\end{proof}

\begin{lemma} $Vol_{\omega_1}(S(U(n)\times U(1)))=\dfrac{\sqrt{n+1}(2\pi)^{\frac{n^2+n}{2}}}{\prod\limits_{i=1}^{n-1}i!}$.
\end{lemma}
\begin{proof} 
Consider the vector $e_1=\begin{pmatrix}
i& 0&\ldots & 0& 0\\
0 & i & \ldots & 0& 0\\
\vdots & & & & \vdots\\
0&0& \ldots &i &0\\
0&0&\ldots &0&-ni 
\end{pmatrix}$ as the first basis vector in the Lie algebra of the group $S(U(1)\times U(n))$. Then the Lie algebra is a direct sum of $\mathbb{C}e_1$ and $su(n)$. The direct product of $SU(n)$ on a one-parameter subgroup, generated by $e_1$ (a circle), is an $n$-sheet cover of $S(U(n)\times U(1))$. The radius of the circle, corresponding to $e_1$, is $\sqrt{(e_1,e_1)}=\sqrt{n^2+n}$. So we get:  
$$Vol_{\omega_1}(S(U(n)\times U(1)))=\dfrac{2\pi \sqrt{n^2+n} \cdot Vol_{\omega_1}(SU(n))}{n}=\dfrac{\sqrt{n+1}(2\pi)^{\frac{n^2+n}{2}}}{\prod\limits_{i=1}^{n-1}i!}.$$
\end{proof}

Now if we substitute $Vol_{\omega_1}(S(U(n)\times U(1)))$ and $\tau_{\infty}(G_{\mathbb{R}}/G_{\mathbb{Z}})$ into the formulas in lemma \ref{itog omega} we get the statement of the theorem \ref{L}.

\section*{The computation of $Vol_{HM}(\mathbb{B}^n/\Gamma')$}
Consider the group $G'=SU(M,K)$. More explicitly, $G'$ is the group of matrices $A$ such that $\begin{cases} AM\bar{A}'=L, \\ \det(A)=1.\end{cases}$ The Lie algebra $\mathcal{L}'$ of the group $G'$ is defined by the system $\begin{cases} BM+M\bar{B}'=0,\\ Tr(B)=0.\end{cases}$ We need to compute new local densities $\tau_2(G'_{\mathbb{Z}_2})$, $\tau_p(G'_{\mathbb{Z}_p})$ for odd $p$ such that $p|D$ and the new determinant of the Killing form. Everything else remains the same as for the lattice $L$ and group $G$.

\begin{lemma}
$|Vol_{\omega^{(d)}}(G'_{\mathbb{R}}/G'_{\mathbb{Z}})|=2^n\cdot d^{\frac{n(n+3)}{4}}\sqrt{n+1}|Vol_{\tau_\infty}(G'_{\mathbb{R}}/G'_{\mathbb{Z}})|$ for $d\equiv 3 \pmod{4}$ and $|Vol_{\omega^{(d)}}(G'_{\mathbb{R}}/G'_{\mathbb{Z}})|=d^{\frac{n(n+3)}{4}}\cdot 2^{\frac{n(n+3)}{2}}\sqrt{n+1}|Vol_{\tau_\infty}(G'_{\mathbb{R}}/G'_{\mathbb{Z}})|$ for $d\equiv 1\pmod{4}$
\end{lemma}
\begin{proof} We fix the basis of $\mathcal{L}'_{\mathbb{Z}}$ over $\mathbb{Z}$. It consists of the elements $g_k$, $e_{i,j}$, $f_{i,j}$ (the same as in $\mathcal{L}_{\mathbb{Z}}$) and elements $e'_{k}=\begin{pmatrix}
0& \ldots & 0 & \ldots& 0\\
\vdots& & \vdots &  & \vdots \\
0 & \ldots & 0&  \ldots & \epsilon \\
\vdots& &\vdots & &\vdots \\
0 & \ldots& 2\bar{\epsilon} & \ldots & 0
\end{pmatrix}$, (instead of $e_k$,non-zero elements are in the $k$-th row of $n+1$-th column and $k$-th column of $n+1$-th row, $k\in\{1...n\}$); 
$f'_{k}=\begin{pmatrix}
0& \ldots & 0 & \ldots& 0\\
\vdots& & \vdots &  & \vdots \\
0 & \ldots & 0&  \ldots & 1\\
\vdots& &\vdots & &\vdots \\
0 & \ldots& 2 & \ldots & 0
\end{pmatrix}$, (instead of $f_k$, non-zero elements are in the $k$-th row of $n+1$-th column and $k$-th column of $n+1$-th row, $k\in\{1...n\}$).

In our basis $\tau_{\infty}=\bigwedge e'_k \bigwedge f'_k \bigwedge g_k \bigwedge e_{i,j}\bigwedge f_{i,j}$.
The matrix of the Killing form in the basis  $\{ g_1,\ldots, g_n,e_{1,2}, f_{1,2},\ldots e_{n-1,n},f_{n-1,n},e'_1, f'_1,\ldots e'_n, f'_n\}$ is:\\

\[
  B_M=\left(\begin{array}{@{}ccccc|cc|c|cc|c@{}}
    -2d & d & 0 & 0 &\ldots & 0 & 0&\ldots &0 &0  &\ldots\\
    d & -2d & d & 0 & \ldots& 0 & 0 &\ldots &0 &0 &\ldots\\
    0 & d& -2d& d&\ldots&  0& 0& \ldots& 0& 0&\ldots\\ 
    \vdots&  \vdots & \vdots & \vdots &\ddots   & \vdots& \vdots&\ddots & \vdots& \vdots& \ddots\\\hline
    0 & 0 & 0 & 0 & \ldots & -2\epsilon \bar{\epsilon} & -\epsilon-\bar{\epsilon} &\ldots &0 &0 &\ldots \\
    0 & 0 & 0 & 0 & \ldots & -\epsilon -\bar{\epsilon} & -2 & \ldots & 0& 0&\ldots\\ \hline
    \vdots & \vdots & \vdots & \vdots  & \ddots & \vdots & \vdots&\ddots & \vdots& \vdots& \ddots \\ \hline
    0 & 0 & 0 & 0 & \ldots & 0& 0& \ldots& 4\epsilon \bar{\epsilon} & 2(\epsilon+\bar{\epsilon}) & \ldots\\
    0 & 0 & 0 & 0 & \ldots & 0& 0& \ldots & 2(\epsilon +\bar{\epsilon}) & 4 & \ldots\\ \hline
    \vdots  & \vdots & \vdots & \vdots & \ddots & \vdots & \vdots& \ddots&\vdots &\vdots & \ddots  \\ \hline
  \end{array}\right)
\]
 So $\sqrt{|\det(B)|}=2^n\cdot d^{\frac{n}{2}}(\epsilon-\bar{\epsilon})^{\frac{n(n+1)}{2}}\sqrt{n+1}$. Clearly, $\sqrt{|\det(B)|}=2^n\cdot d^{\frac{n(n+3)}{4}}\sqrt{n+1}$ for $\epsilon=\frac{1+\sqrt{-d}}{2}$ and $\sqrt{|\det(B)|}=d^{\frac{n(n+3)}{4}}\cdot 2^{\frac{n(n+3)}{2}}\sqrt{n+1}$ for $\epsilon=\sqrt{-d}$. We get the statement of the lemma.
\end{proof}

\begin{lemma}\label{index M}
$[U(M,\mathcal{O}/2^k\mathcal{O}):SU(M,\mathcal{O}/2^k\mathcal{O})]=2|U(1,\mathcal{O}/2^k\mathcal{O})|=\begin{cases} 2^{k+1}(1-\left( \frac{D}{2}\right)2^{-1}), \textit{ 2 is not ramified} \\ 2^{k+2}, \textit{ 2 is ramified}
\end{cases}$ for sufficiently large $k$.
\end{lemma}
\begin{proof}
It is well-known \cite{Otr}, that $|U(1,\mathcal{O}/2^k\mathcal{O})|=\begin{cases} 2^k(1-\left( \frac{D}{2}\right)2^{-1}), \textit{ 2 is not ramified} \\ 2^{k+1}, \textit{ 2 is ramified}
\end{cases}$
 Consider the exact sequence 
$$ 1\to SU(M,\mathcal{O}/2^k\mathcal{O})\to U(M,\mathcal{O}/2^k\mathcal{O})\to \mathcal{U}\to 1,$$
where $\mathcal{U}$ is the set of elements $x$, such that $x$ is the norm of the determinant of matrix in $U(M,\mathcal{O}/2^k\mathcal{O})$. Each element $X\in U(M,\mathcal{O}/2^k\mathcal{O})$ satisfies the equation $XM\bar{X}'=M$, which means that $2Norm(\det(X))=2$. It means that $Norm(\det(X))=1$ or $Norm(\det(X))=1+2^{k-1}$. We note that there exists a matrix $X\in U(M,\mathcal{O}/2^k\mathcal{O})$ such that $Norm(\det(X))=1+2^{k-1}$. For example, we can take $X=diag(1,1,\ldots,1,1+2^{k-2})$. Clearly, $|\mathcal{U}|=2|U(1,\mathcal{O}/2^k\mathcal{O})|$.

\end{proof}

\begin{lemma}
The values of $\tau_2(G'_{\mathbb{Z}_2})$ are listed in the table below:
\begin{table}[h]
\centering
	\begin{tabular}{|c|c|c|}
		\hline
		{\bf n} & {\bf D} & {\bf $\tau_{2}(G'_{\mathbb{Z}_2})$}  \\
	 	\hline
	 	{\bf even} &  {\bf -4d } & { $\frac{1}{2^n}\prod\limits_{i=1}^{\frac{n-2}{2}}(1-2^{-2i})=\tau_{2}(G_{\mathbb{Z}_2})\cdot(1-2^{-n})^{-1}$}    \\
	 	\hline
	 	{\bf even} &  {\bf -d} & { $\prod\limits_{i=2}^{n}(1-\left( \frac{-d}{2}\right)^i 2^{-i} )(1-\left(\frac{-d}{2} \right)2^{-1})=\tau_{2}(G_{\mathbb{Z}_2})\cdot \frac{1-\left(\frac{-d}{2} \right)2^{-1}}{1-\left(\frac{-d}{2} \right)^{n+1}2^{-(n+1)}}$}   \\
	 	\hline
	 	{\bf odd} &  {\bf -4d} & {$\frac{1}{2^n}\prod\limits_{i=1}^{\frac{n-1}{2}}(1-2^{-2i})=\tau_{2}(G_{\mathbb{Z}_2})$ }    \\
	 	\hline
	 	{\bf odd} &  {\bf -d} & {$\prod\limits_{i=2}^{n}(1-\left( \frac{-d}{2}\right)^i 2^{-i} )(1-\left(\frac{-d}{2} \right)2^{-1})=\tau_{2}(G_{\mathbb{Z}_2})\cdot \frac{1-\left(\frac{-d}{2} \right)2^{-1}}{1-\left(\frac{-d}{2} \right)^{n+1}2^{-(n+1)}}$ }   \\
	 	\hline
	\end{tabular}
\end{table}
\end{lemma}
\begin{proof}
We first consider the case $D=-d$. 
We compute $\beta_2$ using the article \cite{Gan}. By definition, $\beta_2=\lim\limits_{N\to\infty}p^{-N(n+1)^2}\#U(M,\mathcal{O}_K/2^N\mathcal{O}_K)$. This limit stabilizes for $N\geq 2\nu_p(\det(M))+1$, so $N=3$. We need to do the following steps:
\begin{enumerate}
    \item Compute $\beta_2$ and $|U(M,\mathcal{O}/8\mathcal{O})|$ using the article \cite{Gan};
    \item Compute $|SU(M,\mathcal{O}/8\mathcal{O})|$ using lemma \ref{index M};
    \item Divide $|SU(M,\mathcal{O}/8\mathcal{O})|$ by $2^{3((n+1)^2-1)}$ (it follows from the proposition \ref{tau_p from balls}).
\end{enumerate}
The Jordan decomposition of $M$ is $M_0\oplus M_1$, $rk(M_0)=n$,  $rk(M_1)=1$ (using the notation from \cite{Gan}). We get $d_0=0,~d_1=1, ~d=1, ~\beta_2=2\cdot 2^{-n^2}\cdot 2^{-1}\cdot|G_0|\cdot|G_1|$. The group $G_0$ is $U(n)$ when $d\equiv  3 \pmod{8}$ and $GL(n)$ when $d\equiv - 1 \pmod{8}$. Similarly $G_1$ is $U(1)$ when $d\equiv  3 \pmod{8}$ and $GL(1)$ when $d\equiv - 1 \pmod{8}$. So, $|G_0|=2^{n^2}\prod\limits_{i=1}^{n}(1-\left( \frac{-d}{2}\right)^i 2^{-i})$, $|G_1|=\begin{cases} 3, ~d\equiv 3 \pmod{8} \\ 1, ~d\equiv -1 \pmod{8}
\end{cases}$. On the other hand, by definition $\beta_2=2^{-3(n+1)^2}|U(M,\mathcal{O}/8\mathcal{O})|$. Using that $[U(M,\mathcal{O}/8\mathcal{O}):SU(M,\mathcal{O}/8\mathcal{O})]=2^4(1-\left( \frac{-d}{2}\right)2^{-1})$, we get  $$|SU(M,\mathcal{O}/8\mathcal{O})|=2^{3((n+1)^2-1)}\prod\limits_{i=1}^{n}(1-\left( \frac{-d}{2}\right)^i 2^{-i}),$$ so $$\tau_2(G'_{\mathbb{Z}_2})=\prod\limits_{i=1}^{n}(1-\left( \frac{-d}{2}\right)^i 2^{-i}).$$ We note that this result doesn't depend on the parity of $n$.

Now consider the case $D=-4d$. We compute $\beta_2$ using the article \cite{Cho}. By definition, $\beta_2=\lim\limits_{N\to\infty}p^{-N(n+1)^2}\#U(M,\mathcal{O}_K/2^N\mathcal{O}_K)$. This limit stabilizes for $N\geq 2\nu_p(\det(M))+3$ (because $p=2$ is ramified), so $N=5$. We need to do the following steps:
\begin{enumerate}
    \item Compute $\beta_2$ and $|U(M,\mathcal{O}/32\mathcal{O})|$ using the article \cite{Cho};
    \item Compute $|SU(M,\mathcal{O}/32\mathcal{O})|$ using lemma \ref{index M};
    \item Consider $\phi:SU(M,\mathcal{O}/32\mathcal{O})\to SU(M,\mathcal{O}/8\mathcal{O})$. It follows from lemma \ref{Hensel} and corollary \ref{Hensel_cor} that $im(\phi)$ coincides with the order of $SU(M,\mathbb{Z}_2)/SU(M,\mathbb{Z}_2)^{(3)}$ (using the notations from proposition \ref{tau_p from balls}). So we need to compute the $|ker(\phi)|$ in order to compute $im(\phi)$.
    \item Divide $im(\phi)$ by $2^{3((n+1)^2-1)}$ (it follows from the proposition \ref{tau_p from balls}).
\end{enumerate}
The Jordan decomposition of $M$ is $M_0\oplus M_2$, $rk(M_0)=n$,  $rk(M_2)=1$. 

We first consider the case when $n$ is even. Using the notation from \cite{Cho} we see that the lattice $M_0$ has type $I^e$, the lattice $M_2$ has type $I^0$, $dim(B_0/Y_0)=n-2$, $dim(B_2/Y_2)=0$. Besides, $\beta=1$, $d_0=0$, $d_2=2\cdot 1 \cdot \frac{1-1}{2}=0$, $N=1$. So,
$$\beta_2=2\cdot 2^{-(n+1)^2}|\tilde{G}|, \textit{ where } |\tilde{G}|=2^m\cdot 2^\beta|Sp(n-2)|, ~m=(n+1)^2-dim(Sp(n-2));$$
$$m=(n+1)^2-\frac{(n-1)(n-2)}{2}, ~|Sp(n-2)|=2^{(\frac{n-2}{2})^2}\prod \limits_{i=1}^{\frac{n-2}{2}}(2^{2i}-1),$$
$$ \beta_2=2^2\prod\limits_{i=1}^{\frac{n-2}{2}}(1-2^{-2i}).$$

By definition $\beta_2=\lim\limits_{N\to \infty}2^{-Ndim(G)}|G'(A/\pi^NA)|$. Since $N=5$ we get $$2^{-5(n+1)^2}|U(M,\mathcal{O}/32\mathcal{O})|=2^2\prod\limits_{i=1}^{\frac{n-2}{2}}(1-2^{-2i}),$$ so $$|U(M,\mathcal{O}/32\mathcal{O})|=2^{5(n+1)^2+2}\prod\limits_{i=1}^{\frac{n-2}{2}}(1-2^{-2i}).$$

Using that $[U(M,\mathcal{O}/32\mathcal{O}):SU(M,\mathcal{O}/32\mathcal{O})]=2^{7}$, we obtain $$|SU(M,\mathcal{O}/32\mathcal{O})|=2^{5(n+1)^2-5}\prod\limits_{i=1}^{\frac{n-2}{2}}(1-2^{-2i}).$$

Now we need to compute $|ker(\phi)|$. It means we need to find the number of solutions $B$ of the system $\begin{cases}
-BM=M\bar{B'}\\
TrB=0
\end{cases}$ modulo $\mathcal{O}/4\mathcal{O}$. Taking $B=(b_{ij})$, we get: $\begin{cases} -b_{ii}=\bar{b}_{ii}, ~i\in\{1;n\} \\  -b_{ij}=\bar{b}_{ji}, ~i,j\in\{1;n\}\\
2b_{i,n+1}=\bar{b}_{n+1,i},~i\in\{1;n\}.
\end{cases}$
Each equation for $b_{ii}$ has $8$ solutions, and in the remaining equations one summand is uniquely determined by the other one, so it gives $16$ options for each pair $(b_{ij}, b_{ji})$. So $|ker(\phi)|=8^n\cdot 16^{\frac{n(n+1)}{2}}=2^{2n^2+5n}$.

So we get $$\tau_2=\frac{|SU(M,\mathcal{O}/32\mathcal{O})|}{|ker(\phi)|\cdot 2^{3((n+1)^2-1)}}=\frac{\prod\limits_{i=1}^{\frac{n-2}{2}}(1-2^{-2i})}{2^n}.$$

When $n$ is odd the proof is similar and the difference occurs only in computation of $\beta_2$. The Jordan decomposition of $M$ is $M_0\oplus M_2$, $rk(M_0)=n$,  $rk(M_2)=1$. Using the notation from the article \cite{Cho} the lattice $M_0$ has type $I^0$, the lattice $M_2$ has type $I^0$, $dim(B_0/Y_0)=n-1$, $dim(B_2/Y_2)=0$. Besides, $\beta=1$, $d_0=0$, $d_2=2\cdot 1 \cdot \frac{1-1}{2}=0$, $N=1$. We get:
$$\beta_2=2\cdot 2^{-(n+1)^2}|\tilde{G}|, \textit{ where } |\tilde{G}|=2^m\cdot 2^\beta|Sp(n-2)|, ~m=(n+1)^2-dim(Sp(n-1));$$
$$m=(n+1)^2-\frac{(n-1)(n)}{2}, ~|Sp(n-1)|=2^{(\frac{n-1}{2})^2}\prod \limits_{i=1}^{\frac{n-1}{2}}(2^{2i}-1),$$
$$ \beta_2=2^2\prod\limits_{i=1}^{\frac{n-1}{2}}(1-2^{-2i}).$$

Everything else is the same as in the case of even $n$, so we get:
 $$\tau_2=\frac{|SU(M,\mathcal{O}/32\mathcal{O})|}{|ker(\phi)|\cdot 2^{3((n+1)^2-1)}}=\frac{\prod\limits_{i=1}^{\frac{n-1}{2}}(1-2^{-2i})}{2^n}.$$
\end{proof}

\begin{lemma}
Let $p$ be odd such that $p|D$. Then $\tau_p(G'_{\mathbb{Z}_p}) = \prod\limits_{i=1}^{\frac{n}{2}}(1-p^{-2i})=\tau_{p}(G_{\mathbb{Z}_p})$ for even $n$ and $\tau_p(G'_{\mathbb{Z}_p})=(1-\left( \frac{(-1)^{\frac{n+3}{2}}\cdot 2}{p}\right) p^{-\frac{n+1}{2}})\prod\limits_{i=1}^{\frac{n-1}{2}}(1-p^{-2i})=\tau_{p}(G_{\mathbb{Z}_p})\cdot\frac{1-\left( \frac{(-1)^{\frac{n+3}{2}}\cdot 2}{p}\right) p^{-\frac{n+1}{2}}}{1-\left( \frac{(-1)^{\frac{n+3}{2}}}{p}\right) p^{-\frac{n+1}{2}}}$ for odd $n$.
\end{lemma}
\begin{proof}
As well as in the proof of lemma \ref{p div D} we notice that $\tau_p(G'_{\mathbb{Z}_p})=\frac{1}{2}\beta_p$ (using the notation from \cite{Gan}). It follows from this article that  $\beta_p=p^{-(n+1)^2}|G_0(\mathcal{O}/p\mathcal{O})|$. If $n$ is even then $G_0$ is an orthogonal group of type I or II, $|G_0(\mathcal{O}/p\mathcal{O})|=2p^{\frac{n^2}{4}}\prod\limits_{i=1}^{\frac{n}{2}}(p^{2i}-1)$. This is the same as for the lattice $L$ so the value of $\tau_p(G'_{\mathbb{Z}_p})$ is also the same. If $n$ is odd then $G_0$ is an orthogonal group of type III or IV, it's order is $|G_0(\mathcal{O}/p\mathcal{O})|=2p^{\frac{n^2-1}{4}}(p^{\frac{n+1}{2}}-\epsilon)\prod\limits_{i=1}^{\frac{n-1}{2}}(p^{2i}-1)$ where $\epsilon=-1$ if $G_0$ is of type III and $\epsilon=1$ if $G_0$ is of type IV. Type III means that the determinant of $M$ is $(-1)^{\frac{n+1}{2}}$, and type IV means that the determinant of $L$ is $\omega(-1)^{\frac{n+1}{2}}$, where $\omega$ is a quadratic nonresidue modulo $p$. Since the determinant of $M$ is $-2$, we get $\epsilon=\left( \frac{(-1)^{\frac{n+3}{2}}\cdot 2}{p} \right)$. For the lattice $L$ the value of corresponding $\epsilon$ was $\left( \frac{(-1)^{\frac{n+3}{2}}}{p} \right)$, so we get the statement of the lemma.
\end{proof}

Now we notice that $$Vol_{HM}(\mathbb{B}^n/\Gamma')=Vol_{HM}(\mathbb{B}^n/\Gamma)\cdot \frac{\tau_2(G_{\mathbb{Z}_2})}{\tau_2(G'_{\mathbb{Z}_2})}\cdot \prod\limits_{p|d}\frac{\tau_p(G_{\mathbb{Z}_p})}{\tau_p(G'_{\mathbb{Z}_p})}\cdot \frac{\det(B_M)}{\det(B_L)}.$$

The proportionality coefficient is listed in the table below. \newpage Substituting it into the above formula we get the statement of theorem \ref{M}.
\begin{table}[h]
\centering
	\begin{tabular}{|c|c|c|}
		\hline
		{\bf n} & {\bf D} & {\bf $\frac{\tau_2(G_{\mathbb{Z}_2})}{\tau_2(G'_{\mathbb{Z}_2})}\cdot \prod\limits_{p|d}\frac{\tau_p(G_{\mathbb{Z}_p})}{\tau_p(G'_{\mathbb{Z}_p})}\cdot \frac{\det(B_M)}{\det(B_L)}$} \\
	 	\hline
	 	{\bf even} &  {\bf -4d } & { $2^n\cdot (1-2^{-n})$}    \\
	 	\hline
	 	{\bf even} &  {\bf -d} & { $2^n\cdot \frac{1-\left(\frac{-d}{2} \right)^{n+1}2^{-(n+1)}}{1-\left(\frac{-d}{2} \right)2^{-1}}$}  \\
	 	\hline
	 	{\bf odd} &  {\bf -4d} & {$2^n\cdot \prod\limits_{p|d}\frac{1-\left( \frac{(-1)^{\frac{n+3}{2}}}{p}\right) p^{-\frac{n+1}{2}}}{1-\left( \frac{(-1)^{\frac{n+3}{2}}\cdot 2}{p}\right) p^{-\frac{n+1}{2}}}$ }   \\
	 	\hline
	 	{\bf odd} &  {\bf -d} & {$2^n\cdot \frac{1-\left(\frac{-d}{2} \right)^{n+1}2^{-(n+1)}}{1-\left(\frac{-d}{2} \right)2^{-1}}\cdot \prod\limits_{p|d} \frac{1-\left( \frac{(-1)^{\frac{n+3}{2}}}{p}\right) p^{-\frac{n+1}{2}}}{1-\left( \frac{(-1)^{\frac{n+3}{2}}\cdot 2}{p}\right) p^{-\frac{n+1}{2}}}$ } \\
	 	\hline
	\end{tabular}
\end{table}


\begin{thebibliography}{1}

\bibitem{Art} E. Artin \textit{Geometric algebra} New York, Interscience Publishers, pp 1-214, 1957

\bibitem{Cho} S. Cho \textit{Group schemes and local densities of ramified hermitian lattices in residue characteristic 2 Part I} Algebra Number Theory 10(3), pp 451-532, 2016

\bibitem{Che} S.S. Chern \textit{On the curvatura integra in a riemannian manifold}, Ann. of Math., Vol 46, N4, pp 674-684, 1945

\bibitem{Gan} W.T. Gan and J.-K. Yu \textit{Group schemes and local densities}, Duke Math. J. 105, pp 497-524, 2000

\bibitem{Hol} R.-P. Holzapfel \textit{Volumes of fundamental domains of Picard modular groups}, Ball and Surface Arithmetics. Aspects of Mathematics, vol 29. Vieweg+Teubner Verlag. 1998, pp 300-329

\bibitem{Mac} I.G. Macdonald \textit{The volume of a compact Lie group} Inventiones Mathematicae, 56(2), pp 93–95, 1980

\bibitem{Otr} G. Otremba \textit{Zur Theorie der hermiteschen Formen in imaginar quadratischen Zahlkorpern}, Journ. f. d. reine und angew. Math. 249, pp 1-19, 1971

\bibitem{Shv} O.V. Schwarzman \textit{On the factor space of an arithmetic discrete group acting on the complex ball}, Thesis, MGU, Moscow, pp 1-77, 1974 Russian)

\bibitem{Zel} H. Zeltinger \textit{Spitzenanzahlen und Volumina Picardscher Modulvarietäten}, Bonner Mathematische Schriften, 136, Universität Bonn, Mathematisches Institut, Bonn, 1981 (German). Dissertation, Rheinische Friedrich-Wilhelms-Universität, Bonn, 1981.

\bibitem{Wei} A. Weil \textit{Adeles and algebraic groups} Progress in Mathematics, 23, Birkhäuser Basel, Springer Science+Business Media New York, pp 126, 1982

\end{thebibliography}
\end{document}